\documentclass[11pt]{elsarticle}

\makeatletter
\def\ps@pprintTitle{%
    \let\@oddhead\@empty
    \let\@evenhead\@empty
    \let\@oddfoot\@empty
    \let\@evenfoot\@oddfoot}
\makeatother

\usepackage{amssymb}
\usepackage{amsmath}

\usepackage{mathrsfs}

\usepackage[top=1.1cm, bottom=1.2cm, left=1.5cm, right=1.5cm]{geometry}

\usepackage{hyperref}

\newtheorem{theorem}{Theorem}[section]
\newtheorem{lem}[theorem]{Lemma}

\newtheorem{prop}[theorem]{Proposition}

\newtheorem{rmk}{Remark}

\newenvironment{proof}{\paragraph{Proof}}{\hfill$\square$}
\numberwithin{equation}{section}
\def \d {\mathrm{d}}

\begin{document}

\begin{frontmatter}



\title{Inverse problems for time-fractional Schr\"odinger equations} 


\author{Salah-Eddine Chorfi$^{*,a}$} 
\ead{s.chorfi@uca.ac.ma}
\author{Fouad Et-tahri$^b$}
\ead{fouad.et-tahri@edu.uiz.ac.ma}
\author{Lahcen Maniar$^{a,c}$}
\ead{maniar@uca.ac.ma}
\author{Masahiro Yamamoto$^{d,e}$}
\ead{myama@ms.u-tokyo.ac.jp}

\affiliation{organization={$^*$Corresponding author. $^a$Cadi Ayyad University, UCA, Faculty of Sciences Semlalia, Laboratory of Mathematics, Modeling and Automatic Systems},
            addressline={B.P. 2390}, 
            city={Marrakesh},
            postcode={40000},
            country={Morocco}}
\affiliation{organization={Faculty of Sciences-Agadir, Lab-SIV, Ibn Zohr University},
            addressline={B.P. 8106}, 
            city={Agadir},
            postcode={80000},
            country={Morocco}}
\affiliation{organization={The UM6P Vanguard Center, Mohammed VI Polytechnic University},
            addressline={Hay Moulay Rachid}, 
            city={Ben Guerir},
            postcode={43150},
            country={Morocco}}
\affiliation{organization={Graduate School of Mathematical Sciences, The University of Tokyo},
            addressline={Komaba, Meguro}, 
            city={Tokyo},
            postcode={153-8914},
            country={Japan}}
\affiliation{organization={Faculty of Science, Zonguldak Bulent Ecevit University},
            city={Zonguldak},
            postcode={67100},
            country={Turkey}}

\begin{abstract}
We study some inverse problems for time-fractional Schr\"odinger equations involving the Caputo derivative of fractional order $\alpha \in (0,1)$. We prove refined uniqueness results from sets of positive Lebesgue measure for various problems by weakening the regularity of initial data.
\end{abstract}



\end{frontmatter}



\section{Introduction and main results}
Throughout the paper, let $0<\alpha<1$, $d\ge 1$ be an integer, $T>0$ be a fixed terminal time, and $\Omega$ be a bounded domain in $\mathbb{R}^{d}$ with sufficiently smooth boundary $\partial \Omega$. We consider the following time-fractional Schr\"odinger equation:
\begin{equation}\label{fseq}
(\mathcal{P}_{f,y_0})\quad
\begin{cases}
\mathrm{i}\,\partial_t^\alpha y(t,x) + (\mathcal{L} y)(t,x)=f(t,x), &\qquad t \in (0,T), \quad x\in \Omega,\\
y(t,x)=0, &\qquad t \in (0,T), \quad x\in \partial\Omega,\\
y(0,x)=y_0(x), &\hspace{2.9cm} x\in \Omega,
\end{cases}
\end{equation}
where $\mathrm{i}$ denotes the imaginary unit satisfying $\mathrm{i}^2=-1$, $\partial_t^\alpha y$ denotes the Caputo fractional derivative of order $\alpha\in (0,1)$, which is defined for a suitable function $y$ by
\begin{equation}\label{cap01}
\partial_t^\alpha y(t)=\frac{1}{\Gamma(1-\alpha)} \int_0^t(t-\tau)^{-\alpha}\partial_\tau y(\tau)\, \d \tau.
\end{equation}
Above, $\Gamma$ denotes the Euler Gamma function. We assume that the operator $\mathcal{L}$ is symmetric and uniformly elliptic on $\overline{\Omega}$; see Section \ref{sec2} for the details. Moreover, $f$ is a source term, and $y_0$ is an initial datum.

In \cite{Na04}, Naber derived the time-fractional Schr\"odinger equation by mapping the time-fractional diffusion equation and introducing the factors $\mathrm{i}$ or $\mathrm{i}^\alpha$ in \eqref{fseq}. On the other hand, Achar et al. \cite{Ach13} obtained \eqref{fseq} via the Feynman path integral. This model allows for the recovery of the standard Schr\"odinger equation as $\alpha\to 1^-$. Such equations are used in various areas of physics, including quantum mechanics, optics, and plasma physics; see, e.g., \cite{CC24, CC23, La18} and the cited bibliographies. We emphasize that the literature on inverse problems for such models is sparse, and we can only refer to \cite{AS23, CEMY25} and the references therein. It is worth noting that, while our analysis focuses on \eqref{fseq} with the factor $\mathrm{i}$, it also applies to the formulation involving $\mathrm{i}^\alpha, \; \alpha\in (0,1)\cup (1,2)$.

Henceforth, $E \subset \Omega$ will denote an arbitrary set of positive measure, and the equality of functions on a measurable set refers to equality almost everywhere.

\textbf{Inverse initial data problem.} Determine the initial datum $y_0$ from the knowledge of the measurement $y$ in $(0,T)\times E$ of the solution to $(\mathcal{P}_{0,y_0})$.

The first result concerns the uniqueness for initial data in 
$(\mathcal{P}_{0,y_0})$. 
\begin{theorem}\label{thm1}
Let $y_0 \in L^2(\Omega)$ and $y$ satisfy $(\mathcal{P}_{0,y_0})$. Then $y=0$ in $(0, T)\times E$ implies $y_0=0$ in $\Omega$.
\end{theorem}
This theorem improves the corresponding uniqueness in \cite[Theorem 4.2]{SY11} 
concerning the time-fractional diffusion equations in two directions: 
(i) it removes the regularity of initial data (needed for the Sobolev embedding) and (ii) it holds for a subset of positive measure instead of a subdomain. An application of this result in control theory is the approximate controllability of the adjoint problem of $(\mathcal{P}_{0,y_0})$ by controls acting on any subset of positive measure.

\begin{rmk} Some remarks are in order
\begin{itemize}
    \item[(i)] In the 1D case $d=1$, it suffices for $E$ to be only 
an infinite set, since each eigenfunction has only finitely many zeros. Moreover, for $\Omega=(0,1)$ and $\mathcal{L}=\frac{\d^2}{\d x^2}$, all the eigenvalues are simple. If for some $x_0\in (0,1)\setminus \mathbb{Q}$, $y(t,x_0)=0, \; t\in (0,T)$, then $y_0=0$ in  $(0,1)$.
    \item[(ii)] In some 2D cases, one can obtain the uniqueness from a discrete set; see, e.g., \cite[Theorem 3.2]{Bu93}.
    \item[(iii)] Sets of positive Lebesgue measure are not the sharpest possible. One can consider sharper sets with positive Hausdorff measure as in 
\cite[Th\'eor\`eme 2]{Ro88}.
\end{itemize}

\end{rmk}

Now, we consider some related inverse problems:

\textbf{Inverse source problem.} We assume that $f(t,x)=\rho(t)g(x)$, where $\rho$ is a given temporal factor. Determine the spatial component $g$ in $\Omega$ from the knowledge of the measurement $y$ in $(0,T)\times E$ of the solution to $(\mathcal{P}_{\rho g,0})$.

\begin{theorem}\label{thm2}
Let $g \in L^2(\Omega)$ and $\rho\in L^2(0,T)$ such that $\rho\not\equiv 0$ in $(0,T)$. Let $y$ satisfy $(\mathcal{P}_{\rho g,0})$. Then $y=0$ in $(0, T)\times E$ implies $g=0$ in $\Omega$.
\end{theorem}

Let $y_{\alpha, u}$ be the solution to $(\mathcal{P}_{0,u})$ associated with order $\alpha \in (0,1)$ and initial datum $u\in L^2(\Omega)$.

\textbf{Inverse order problem}. Determine the fractional order $\alpha \in (0,1)$ from the knowledge of the measurement $y_{\alpha,u}$ in $(0,T)\times E$.

The following uniqueness result is proved similarly to \cite[Theorem 1]{Ya21}, so the proof is omitted.
\begin{theorem}\label{thm3}
Let $0<\alpha,\beta<1$ and $u,v\in L^2(\Omega)$ such that $u\neq 0$ in $\Omega$.
Then, $y_{\alpha, u}=y_{\beta, v} \quad \text{ in }\; (0,T)\times E$ implies $\alpha=\beta$. 
\end{theorem}

Note that in the aforementioned inverse problems, we can instead consider boundary measurement given by the Neumann data $\partial_{\nu}^a y\rvert_{(0,T)\times\gamma},$ on an arbitrarily chosen subboundary $\gamma \subset \partial\Omega$.

\section{Preliminary results}
\label{sec2}
We mainly work in the Lebesgue-space $L^{2}(\Omega)$ of $\mathbb{C}$-valued functions with the standard inner product $\langle\cdot, \cdot\rangle$, as well as the usual $L^2$-based Sobolev spaces $H^{k}(\Omega)$ and $H_{0}^{k}(\Omega)$.

Henceforth, we shall consider a symmetric and uniformly elliptic operator $\mathcal{L}$ of the form
$$
\mathcal{L} y(x)=\sum_{i,j=1}^{d} \partial_i\left(a_{i j}(x) \partial_j y(x)\right)+ p(x) y(x), \quad x \in \Omega,
$$
where the coefficients are real-valued and satisfy
$$
a_{i j} \in W^{1,\infty}(\Omega),\quad a_{i j}=a_{j i},\; 1 \leq i, j \leq d \quad\text{ and } \quad p \in L^\infty(\Omega),
$$
and there exists a constant $\kappa >0$ such that
$$
\sum_{i, j=1}^{d} a_{i j}(x) \xi_{i} \xi_{j} \ge \kappa \sum_{i=1}^{d} \xi_{i}^{2}, \quad x \in \overline{\Omega},\; \xi \in \mathbb{R}^{d}.
$$
The domain of the operator $-\mathcal{L}$ is given by $\mathcal{D}(-\mathcal{L})=H^{2}(\Omega) \cap H_{0}^{1}(\Omega).$ Since the operator $-\mathcal{L}$ is symmetric and uniformly elliptic, its spectrum $\sigma\left(-\mathcal{L}\right)$ is entirely composed of eigenvalues $\{\lambda_n\}_{n\in \mathbb{N}}$ counted according to the multiplicities. Without loss of generality, we may assume that $p\ge 0$. Then, we can set
$$0<\lambda_{1} \leq \lambda_{2} \leq \cdots \leq \lambda_n \le \cdots.$$
By $\{\varphi_{n}\}_{n\in \mathbb{N}}, \;\varphi_{n}\in H^{2}(\Omega) \cap H_{0}^{1}(\Omega),$ we denote the corresponding Hilbert basis of $L^2(\Omega)$ 
composed of eigenfunctions.

Next, we recall the Mittag-Leffler function defined by
$$
E_{\alpha, \beta}(z):=\sum_{k=0}^{\infty} \frac{z^{k}}{\Gamma(\alpha k+\beta)}, \quad z \in \mathbb{C},
$$
where $\alpha>0$ and $\beta \in \mathbb{R}$ are arbitrary parameters. We will use the following boundedness property of the Mittag-Leffler functions (see \cite[Theorem 1.6]{Pod99}).
\begin{lem}\label{mlest}
Let $0<\alpha<1$ and $\lambda\ge 0$. There exists a constant $C_0>0$, depending only on $\alpha$ and $\mu \in\left(\frac{\pi \alpha}{2}, \pi \alpha \right)$, such that
\begin{equation}\label{mlestt}
\left|E_{\alpha, 1}\left(-\mathrm{i} \lambda z^\alpha\right)\right| \leq \frac{C_0}{1+\lambda |z|^\alpha} \leq C_0, \qquad \mu \le |\arg (-\mathrm{i} z^\alpha)| \le \pi.
\end{equation}
\end{lem}

We have the following well-posedness result.
\begin{prop} \label{FP}
Let $0<\alpha < 1$, $y_0\in L^{2}(\Omega)$ and $f\in L^1(0,T;L^2(\Omega))$. Then, there exists a unique solution $y \in L^1(0,T;L^{2}(\Omega))\cap C\left((0,T];L^{2}(\Omega)\right)$ to $(\mathcal{P}_{f,y_0})$ such that 
the following estimate holds
	\begin{align*}
		&\|y\|_{L^1(0,T;L^{2}(\Omega))}\le C \left(\|y_0\|_{L^2(\Omega)}+\|f\|_{L^1(0,T;L^2(\Omega))}\right)
	\end{align*}
for some constant $C>0$. The unique solution to $(\mathcal{P}_{f,y_0})$ is given by the formula
\begin{equation}\label{solution formula}
	y(t,\cdot)=\sum_{n=1}^{\infty}\left[\left\langle y_0, \varphi_n\right\rangle E_{\alpha, 1}\left(-\mathrm{i}\lambda_n t^\alpha\right) -\mathrm{i}\int_{0}^{t}\left\langle f(s,\cdot), \varphi_n\right\rangle (t-s)^{\alpha-1}E_{\alpha,\alpha}(-\mathrm{i}\lambda_n(t-s)^{\alpha}))\,\d s\right] \varphi_n.
\end{equation}
Moreover, the unique solution to $(\mathcal{P}_{0,y_0})$ is analytically 
extendable in $L^{\infty}(\Sigma; L^2(\Omega))$, where $\Sigma$ is a 
sector defined by 
\begin{equation}\label{sector}
\Sigma=\left\{z \in \mathbb{C}\setminus\{0\} \colon \max\left(-\pi, \frac{\mu-\frac{3\pi}{2}}{\alpha}\right) \le \arg z \le \min\left(\pi, \frac{\frac{\pi}{2}-\mu}{\alpha}\right)\right\},
\end{equation}
and $\mu$ is chosen as in Lemma \ref{mlest}.
\end{prop}
\begin{proof}
The existence and the representation formula are classical, and can be obtained from \cite{CEMY25} in a similar way to \cite{SY11}. 
Regarding the analyticity, we know that $E_{\alpha, 1}\left(-\mathrm{i}\lambda_{n} z\right)$ is an entire function; see, e.g., \cite[Proposition 3.1]{Gor20}. Then, $E_{\alpha, 1}\left(-\mathrm{i}\lambda_{n} z^{\alpha}\right)$ is analytic 
in
$$
\Sigma=\{z \in \mathbb{C}\setminus\{0\} \colon \mu \le |\arg (-\mathrm{i} 
z^\alpha)| \le \pi\}.
$$
Thus, $y_{N}(z,\cdot)=\displaystyle\sum_{n=1}^{N}\langle y_0, \varphi_{n}\rangle E_{\alpha, 1}\left(-\mathrm{i}\lambda_{n} z^{\alpha}\right) \varphi_{n}$ is 
analytic in $\Sigma$, and by Lemma \ref{mlest}, we have 
\begin{equation}
\left\|y_{N}(z,\cdot)-y(z,\cdot)\right\|_{L^{2}(\Omega)}^{2}\leq C_{0}^2 \sum_{n=N+1}^{\infty}\left|\langle y_0, \varphi_{n}\rangle\right|^{2}, \quad z \in \Sigma.
\end{equation}
Therefore, $\lim\limits_{N \to \infty}\left\|y_{N}-y\right\|_{L^{\infty}\left(\Sigma; L^{2}(\Omega)\right)}=0$, which yields that also $y$ is analytic 
in $\Sigma$. 
Now, setting $\theta:=\arg z \in(-\pi, \pi]$, we have
$$\begin{aligned}
&\arg\left(-\mathrm{i} z^\alpha\right)= \begin{cases}
\alpha \theta-\dfrac{\pi}{2}, & \theta>-\dfrac{\pi}{2 \alpha},\\[2ex]
\alpha \theta+\dfrac{3 \pi}{2}, & \theta \leq-\dfrac{\pi}{2 \alpha}.
\end{cases}
\end{aligned}$$
Hence, by elementary calculations, we can prove that $\Sigma$ is given by \eqref{sector}.
\end{proof}

\section{Proofs of the main results}
The next proof is a modification of its counterpart in \cite{SY11}.
\subsection{\textbf{Proof of Theorem \ref{thm1}}}
We keep the same notations as in the proof of Proposition \ref{FP}. Since $y$ is analytic on the sector $\Sigma$, which contains the positive real axis for a choice $\mu<\frac{\pi}{2}$, we obtain
\begin{equation}\label{eqanE}
y(t,\cdot)=\sum_{n=1}^{\infty}\langle y_0, \varphi_{n}\rangle E_{\alpha, 1}\left(-\mathrm{i} \lambda_{n} t^{\alpha}\right) \varphi_{n}=0 \quad \mbox{in}\;\; L^2(E),\quad t>0. 
\end{equation}
On the other hand, using \eqref{mlestt}, we obtain for any $z\in \mathbb{C}$ and $t>0$,
$$
\left\|e^{-zt}y_{N}(t,\cdot)\right\|_{L^2(E)}\le C_0  
e^{-\operatorname{Re}(z)t} \|y_0\|_{L^2(\Omega)}=:\nu(t)\quad 
\quad \mbox{for all $t>0$ and all $N\in \mathbb{N}$}.
$$
Note that $\nu\in L^1(0,\infty)$ if $\operatorname{Re}(z)>0$. Then, applying Lebesgue's dominated convergence theorem for functions with values in $L^2(E)$ and using \eqref{eqanE} yields
\begin{equation}\label{AA}
	\displaystyle\sum_{n=1}^{\infty}\langle y_0, \varphi_{n}\rangle \left[\int_{0}^{\infty}e^{-zt} E_{\alpha, 1}\left(-\mathrm{i}\lambda_{n} t^{\alpha}\right)\,\d t\right] \varphi_{n}
	= \int_{0}^{\infty}e^{-zt}y(t,\cdot)\,\d t=0 \quad\mbox{in}\;\; L^{2}(E),\;\; \operatorname{Re}(z)>0.
\end{equation}
Now, by grouping the eigenfunctions corresponding to the same eigenvalue, we denote by
$\sigma(-\mathcal{L})=\left\{\mu_{k} \right\}_{k \in \mathbb{N}}$ the strictly increasing sequence of eigenvalues of $-\mathcal{L}$ and by $\left\{\varphi_{k j}\right\}_{1 \leq j \leq m_{k}}$ an orthonormal basis of $\ker\left(\mu_{k}+\mathcal{L}\right)$. 
Therefore, \eqref{AA} can be rewritten as
\begin{equation}\label{BB}
\sum_{k=1}^{\infty}\left(\left[\int_{0}^{\infty}e^{-zt} E_{\alpha, 1}\left(-\mathrm{i}\mu_{k} t^{\alpha}\right)\,\d t\right]\sum_{j=1}^{m_{k}}\langle y_0, \varphi_{kj}\rangle \varphi_{k j} \right)=0 \quad \mbox{in}\;\; L^{2}(E),\;\; \operatorname{Re}(z)>0.
\end{equation}
We use the Laplace transform (see \cite[formula (1.80)]{Pod99}) and the 
analytic continuation to obtain
\begin{equation} \label{ltf}
\int_{0}^{\infty} e^{-z t} E_{\alpha, 1}\left(-\mathrm{i} \mu_{k} t^{\alpha}\right)\, \d t=\frac{z^{\alpha-1}}{z^{\alpha}+\mathrm{i}\mu_{k}}, \quad \operatorname{Re} z>0.
\end{equation}
Therefore, \eqref{BB} and \eqref{ltf} imply
\begin{equation}\label{eqac}
\sum_{k=1}^{\infty} \sum_{j=1}^{m_{k}}\langle y_0, \varphi_{kj}\rangle \frac{1}{\eta+\mathrm{i}\mu_{k}} \varphi_{k j}=0 \quad \quad\mbox{in}\;\; L^{2}(E),\; \operatorname{Re} \eta>0.
\end{equation}
For all $n\in \mathbb{N}$, we define
\begin{align*}
	S_n(\eta,\cdot):=&\sum_{k=1}^{n} \sum_{j=1}^{m_{k}}\langle y_0, \varphi_{kj}\rangle \frac{1}{\eta+\mathrm{i}\mu_{k}} \varphi_{k j} \quad \quad\mbox{in}\;\; L^{2}(\Omega),\; \eta\in\mathbb{C}\setminus\{-\mathrm{i}\mu_k \colon k\ge 1\}.
\end{align*}
Since $\{\mu_k\}$ tends to $\infty$, the set $\{-\mathrm{i}\mu_k \colon k\ge 1\}$ is closed. Then, for every compact subset $K$ of $\mathbb{C}\setminus\{-\mathrm{i}\mu_k \colon k\ge 1\}$, we have
$$\delta:=d(K, \{-\mathrm{i}\mu_k \colon k\ge 1\})>0.$$
Let $n>m\ge 1$. Then 
\begin{align}
	\left\|S_{n}(\eta,\cdot)-S_{m}(\eta,\cdot)\right\|_{L^{2}(\Omega)}^{2}
	&\le \delta^{-2}\sum_{k=m+1}^{n}\sum_{j=1}^{m_k}\left|\langle y_0, \varphi_{k_j}\rangle\right|^{2}, \quad \eta \in K. \label{CS}
\end{align} 
Therefore, for every $\eta\in K$, $\{S_n(\eta,\cdot)\}$ is a Cauchy sequence in $L^2(\Omega)$. Hence, its sum 
$S$ is well-defined on $\mathbb{C}\setminus\{-\mathrm{i}\mu_k \colon k\ge 1\}$:
\begin{align*}
	S(\eta,\cdot):=&\sum_{k=1}^{\infty} \sum_{j=1}^{m_{k}}\langle y_0, \varphi_{kj}\rangle \frac{1}{\eta+\mathrm{i}\mu_{k}} \varphi_{k j} \quad \quad\mbox{in}\;\; L^{2}(\Omega),\; \eta\in \mathbb{C}\setminus\{-\mathrm{i}\mu_k \colon k\ge 1\}.
\end{align*}
Then, by \eqref{CS}, we obtain that $\{S_n\}$ converges uniformly to $S$ on $K$. Hence, 
$S$ is analytic in $\mathbb{C}\setminus\{-\mathrm{i}\mu_k \colon k\ge 1\}$. It follows from \eqref{eqac} that
 \begin{equation}\label{eqac2}
 	S(\eta,\cdot)=\sum_{k=1}^{\infty} \sum_{j=1}^{m_{k}}\langle y_0, \varphi_{kj}\rangle \frac{1}{\eta+\mathrm{i}\mu_{k}} \varphi_{k j}=0 \; \quad\mbox{in}\;\; L^{2}(E),\; \eta\in \displaystyle \mathbb{C}\setminus\{-\mathrm{i}\mu_k \colon k\ge 1\}. 
 \end{equation}
Now, we fix $\ell\in \mathbb{N}$. By the uniform convergence, we can integrate \eqref{eqac2} termwise on the positively oriented circle $\mathcal{C}_{\ell}:=\mathcal{C}(-\mathrm{i}\mu_\ell, r_l)$ with radius $r_\ell>0$ chosen such that $-\mathrm{i}\mu_k\notin \overline{\mathcal{C}_{\ell}}$ for all $k\neq \ell$. Then, we obtain
$$
y_{\ell} := \sum_{j=1}^{m_{\ell}}\langle y_0, \varphi_{\ell j}\rangle \varphi_{\ell j}=\frac{1}{2\pi \mathrm{i}}\int_{\mathcal{C}_{\ell}}S(\eta,\cdot)\,\d\eta=0 \quad \mbox{in}\;\; L^{2}(E).
$$
Since $\left(\mathcal{L}+\mu_{\ell}\right) y_{\ell}=0$ in $L^2(\Omega)$ and $y_{\ell}=0$ in $L^{2}(E)$, we use the unique continuation from subsets of positive measure. See, e.g., \cite[Th\'eor\`eme 2]{Ro88}. This implies that $y_{\ell}=0$ in $L^2(\Omega)$ for all $\ell \in \mathbb{N}$. Therefore, we conclude that $y_0=0$ in $L^2(\Omega)$. \qed

As in Theorem \ref{thm1}, we can also improve \cite[Theorem 4.3]{SY11} for fractional diffusion equations with an $L^{\infty}$-decay estimate. We have the following uniqueness result.
\begin{prop}
Let $y_0\in L^2(\Omega)$ and assume that for any $m \in \mathbb{N}$, there exists a constant $C(m)>0$ such that the solution of $(\mathcal{P}_{0,y_0})$, considered on $(0,\infty)$, satisfies
\begin{equation}\label{decay}
\|y(t,\cdot)\|_{L^{2}(E)} \leq \frac{C(m)}{t^{m}} \quad \text { as } t \to \infty.
\end{equation}
Then $y=0$ in $(0, \infty)\times \Omega$.
\end{prop}
The proof is similar to \cite[Theorem 4.3]{SY11}, but one should use $L^2$-estimates instead of pointwise estimates.

\subsection{\textbf{Proof of Theorem \ref{thm2}}}
For $\beta > 0$, we set 
$$
J^{\beta} w(t):= \frac{1}{\Gamma(\beta)}\int^t_0 (t-s)^{\beta-1}w(s)\,\d s, \quad w\in L^1(0,T;L^2(\Omega)).
$$
Let $y(g)$ be the solution of $(\mathcal{P}_{\rho g,0})$ and $v(g)$ be the solution of $(\mathcal{P}_{0,g})$. 
Similarly to the time-fractional diffusion equation (e.g., \cite{LRY17}), 
we can prove Duhamel's principle for the time-fractional Schr\"odinger 
equation, and so we have
$$
J^{1-\alpha} y(g)(t,\cdot) = \int^t_0 \rho(t-s)v(g)(s,\cdot)\,\d s \quad 
\mbox{in}\;\; L^2(\Omega), \; 0<t<T.
$$
Since $y(g) = 0$ on $(0,T) \times E$, we obtain
\begin{equation}\label{eq(2.1)}
\int^t_0 \rho(t-s)v(g)(s,\cdot)\,\d s = 0 \quad
\mbox{in}\;\; L^2(E), \; 0<t<T.
\end{equation}
Choosing $\theta \in C^{\infty}_0(E)$ arbitrarily, we set 
$w_{\theta}(t):= \langle v(g)(t,\cdot),\, \theta\rangle_{L^2(E)}$.
Then, \eqref{eq(2.1)} yields 
$$
\int^t_0 \rho(t-s)w_{\theta}(s)\,\d s = 0, \quad 0<t<T.
$$
In view of $\rho \not\equiv 0$ in $(0,T)$, the Titchmarsh's 
convolution theorem (\cite[Theorem VII]{Ti26}) implies
that we can find a constant $t_*=t_*(\theta) \in (0,T]$ such that 
\begin{equation}\label{eq(2.2)}
w_{\theta}(t) = 0 \qquad \text{for all }\; 0<t<t_*(\theta) \quad (w_{\theta} \text{ is continuous}).
\end{equation}
By the analyticity of $v(g): (0,\infty) \mapsto L^2(\Omega)$, it follows 
that $w_{\theta}$ is analytic in $t>0$. Hence, \eqref{eq(2.2)} yields
\begin{equation}\label{eq(2.3)} 
\langle v(g)(t,\cdot),\, \theta\rangle_{L^2(E)} = 0 \quad \text{ for all }t \in (0,\infty).
\end{equation}
Since $C^{\infty}_0(E)$ is dense in $L^2(E)$, then \eqref{eq(2.3)} implies
$$
v(g)(t,\cdot) = 0 \quad \text{ in } L^2(E), \quad \text{ for all }t \in (0,\infty).
$$
Therefore, Theorem \ref{thm1} implies that $g=0$ in $\Omega$.

\begin{rmk}
We emphasize that our approach applies to any self-adjoint operator on $L^2(\Omega)$ possessing a Hilbert basis of eigenfunctions that satisfy the unique continuation property. Consequently, it can be used in various settings with different boundary conditions.
\end{rmk}

\section*{Funding}
The work was supported by Grant-in-Aid for Challenging Research (Pioneering) 
21K18142 of Japan Society for the Promotion of Science.

\section*{Declaration of competing interest}
The authors declare that they have no known competing interests that could have appeared to influence the work reported in this paper.

\section*{Data availability}
No data was used for the research described in the article.

\end{document}